\documentclass[12pt]{amsart}
\usepackage{amsmath, amssymb, amsthm}
\usepackage{comment, empheq}
\usepackage{mathdots, breqn}
\usepackage{graphicx}
\usepackage{hyperref}
\usepackage{mathtools}
\numberwithin{equation}{section}
\newtheoremstyle{myremark}{10pt}{10pt}{}{}{\scshape}{.}{.5em}{}
\newtheorem{theorem}{Theorem}
\newtheorem{lemma}{Lemma}[section]

\theoremstyle{remark}
\theoremstyle{myremark}

\newtheorem*{ackno}{Acknowledgements}
\newcommand{\HH}{\mathbb{H}}
\newcommand{\Q}{\mathbb{Q}}
\newcommand{\R}{\mathbb{R}}
\newcommand{\C}{\mathbb{C}}
\renewcommand{\P}{\mathcal{P}}
\newcommand{\M}{\mathcal{M}}
\newcommand{\OO}{\mathcal{O}}
\newcommand{\Z}{\mathbb{Z}}
\newcommand{\hh}{\mathcal{H}}
\newcommand{\N}{\mathcal{N}}
\newcommand{\GL}{\mathrm{GL}}
\newcommand{\PGL}{\mathrm{PGL}}
\newcommand{\SL}{\mathrm{SL}}

\begin{document}

\title{Supnorm of an eigenfunction of finitely many Hecke operators}

\author{Subhajit Jana}

\address{ETH Z\"urich, Switzerland.}
\email{subhajit.jana@math.ethz.ch}

\maketitle

\begin{abstract}
Let $\phi$ be a Laplace eigenfunction on a compact hyperbolic surface attached to an order in a quaternion algebra.  Assuming that $\phi$ is an eigenfunction of Hecke operators at a \emph{fixed finite} collection of primes, we prove an $L^\infty$-norm bound for $\phi$ that improves upon the trivial estimate by a power of the logarithm of the eigenvalue. We have constructed an amplifier whose length depends on the support of the amplifier on Hecke trees. We have used a method of B\'erard in \cite{Be} to improve the archimedean amplification.

\end{abstract}

\section{Introduction}
\subsection{Main Theorem}
Let $M$ be a compact Riemannian manifold and $\phi$ be a function on $M$ that satisfies $\Delta\phi+\lambda^2\phi=0$ and $\|\phi\|_2=1$, where $\Delta$ is the Laplace-Beltrami operator on $M$ and by $\|\phi\|_p$ we mean $\|\phi\|_{L^p(M)}$. It is known that (see e.g. \cite{Ho}) one can have a general bound of the form
\begin{equation}\label{start}
\|\phi\|_{\infty}\ll\lambda^{\nu(M)} \;\;\;\;\; \text{ with } \;\;\;\;\;  \nu(M)=\frac{\dim(M)-1}{2}.
\end{equation}
It is usually believed that if the geodesic flow on the unit cotangent bundle of $M$ is sufficiently chaotic then the bound \eqref{start} can be improved. One result of this kind is due to B\'erard \cite{Be}, who proves that if $M$ has negative sectional curvature, then one has
\begin{equation}\label{startB}
\|\phi\|_\infty\ll\frac{\lambda^{\nu(M)}}{\sqrt{\log\lambda}}.
\end{equation}

The case of bounding the sup-norm of eigenfunctions on arithmetic manifolds is an interesting and more tractable problem due to automorphy and symmetries arising from the underlying Hecke algebra; one expects an improvement in \eqref{start}, namely $\|\phi\|_\infty\ll\lambda^\epsilon$, on a compact hyperbolic surface. The first breakthrough was achieved by Iwaniec and Sarnak \cite{IS} who proved that on a compact arithmetic hyperbolic surface (such as a quotient of $\hh^2$ by the group of units in an order of a quaternion division algebra over $\Q$) if $\phi$ is an eigenfunction of the \emph{full} Hecke algebra then $$\|\phi\|_\infty\ll_\epsilon\lambda^{\frac{5}{12}+\epsilon}.$$ While the most general question still remains open, one can weaken the assumption of  full Hecke symmetry and ask to improve upon the trivial estimate \eqref{startB}.

In this article, we will improve the estimate \eqref{startB} assuming only that $\phi$ is an eigenfunction of $p$-Hecke operators for \emph{fixed finitely many} primes $p$. Following the pre-trace formula approach as in \cite{IS} we achieve a \emph{power of logarithm} saving in \eqref{startB}. We will describe now our main theorem.

Let $M$ be a compact hyperbolic surface. From the theorem of Cartan and Hadamard and the uniformization theorem, $M=\Gamma\backslash\hh^2$, where $\hh^2$ is the hyperbolic plane and $\Gamma$ is a discrete subgroup of $\SL_2(\R)$. To make use of the Hecke symmetry we assume that $\Gamma$ is a discrete subgroup of $\SL_2(\R)$ with an arithmetic structure. For the ease of presentation of the paper we will henceforth assume that $\Gamma$ is the unit group of a maximal order in a quaternion division algebra over $\Q$. Let $\P$ be a fixed finite set of primes at which  the underlying order is unramified.

\begin{theorem}\label{theorem1}
Let $\phi$ simultaneously be a Maass form on $M$ with Laplace eigenvalue $\lambda^2$ and a Hecke eigenform of the $p$-adic Hecke algebra for all $p\in\P$. Then
$$\|\phi\|_\infty\ll_\P
\frac{{\lambda^{1/2}}}{{(\log\lambda)^{\frac{|\P|+1}{2}}}},$$
where $|\P|$ denotes the cardinality of $\P$.
\end{theorem}

There are several results of the Iwaniec-Sarnak-type in higher dimensional and higher rank arithmetic locally symmetric spaces (see e.g. \cite{Sa1}, \cite{BHM}, \cite{BHMM}, \cite{HRR}, \cite{BM1}, \cite{BM2}, \cite{BP}, \cite{Ma2} and their references). It is by now well-known to the  experts that the idea behind saving a power
of the eigenvalue in the supnorm problem for a Maass eigenform of several Hecke
operators lies in the number of returns that the Hecke operators make near the maximal
compact (see \cite{Ma2} for elaborate discussion about the criticality of the size of the maximal compact subgroup). As far we know, in all previous literature regarding this problem, the authors have always assumed that the Maass form is also an eigenform of $p$-Hecke operators for infinitely many distinct primes $p$. This assumption enlarges the scope of use of Hecke symmetry and leads to a better control on the Hecke returns. Here we are assuming that the Maass form is an eigenform of finitely many $p$-adic Hecke algebras. However, this weaker hypothesis allows us a weaker result than the result of \cite{IS}.

\subsection{Remarks}
\begin{enumerate}
\item One can compare \eqref{startB} and the bound in Theorem \ref{theorem1} as the Hecke operator at prime $p$ is saving $\sqrt{\log\lambda}$ from the former estimate. This is a non-archimedean analogue of the fact that each differential operator in the universal enveloping algebra of a locally symmetric space helps saving $\lambda^{1/2}$ from \eqref{start} (see \cite{Sa1}).
\item
The general idea of the proof of Theorem \ref{theorem1} is the same as in \cite{IS}. But as we have, unlike \cite{IS}, used the Hecke symmetry at only finitely many places, we could no longer use the usual Duke-Friedlander-Iwaniec-type amplifier which is supported on a fixed sphere on various Hecke-trees. On the other hand, inspired by \cite{N}, here we are using a different kind of amplifier which is supported on spheres of growing radii. The success of this amplifier depends on the size of its radius, whereas in the former case, it depended on the number of primes contributing to the amplifier. A similar kind of amplifier has been constructed in \cite{BL} where the authors used the one-prime amplifier to prove a positive entropy result in the context of the quantum unique ergodicity; that essentially relies on a counting problem, as in ours, although in a different group.
\item
The bound we have proved in Theorem \ref{theorem1} should be optimal among the savings one can get following the Iwaniec-Sarnak method and using an amplifier which is constructed by taking a linear combination of Hecke operators. For details one may look at the remark in \S $4.3$.
\item
It can be understood from the proof that we can get a larger saving (as big as $e^{-c\sqrt{\log\lambda}}$ for $c>0$) if $\phi$ is non-tempered at a certain prime, however, we do not discuss it here.
\item A similar approach can be followed to get a logarithm saving for the sup norm of Maass cusp forms which are eigenforms of finitely many Hecke operators on congruence quotients of hyperbolic 3-spaces.
\end{enumerate}

We conclude the Introduction with remarking that the supnorm problem is essentially connected to the problem of multiplicity of the spectrum of the Laplacian (see \cite{Sa1}). One expects that on an arithmetic surface arising from the unit group of a maximal order (e.g. $M$), the spectrum of the Laplacian is essentially simple; in that case the improved result of \cite{IS} can be extended to general Maass forms.

\section{Non-Archimedean Amplification}
\subsection{p-adic structure} In this subsection we introduce the Hecke operators to construct the amplifier. For the standard facts in the Hecke theory one may look at \cite{E}.

Let $\HH$ be a quaternion division algebra over $\Q$ and $\OO$ a maximal $\Z$-order in $\HH$ and we fix an embedding $\tau$ of $\HH$ into the $2\times 2$ matrix group over a suitable number field. We will always assume $p$ is a prime in $\Z$ such that $\HH$ splits at $p$ (which happens for all but finitely many primes). Let $\N$ be the reduced norm on $\HH$ defined by $\N(x)=x\bar{x}$.

Let us define $$R(m)=\{\gamma\in \OO\mid \N(\gamma)=m\};$$we define $\Gamma=\tau(R(1))< \SL_2(\R).$ We define the $n$th Hecke operator acting on the space of functions on $M=\Gamma\backslash\hh^2$ by 
$$T(n)f(z)=\frac{1}{\sqrt{n}}\sum_{\gamma\in R(1)\backslash R(n)}f(\gamma z).$$ 
If $m,n$ are coprime with the non-split primes of the division algebra then one has the Hecke relation (see $0.6$ of \cite{IS}, also Ch. 3 of \cite{E}),
\begin{equation}\label{Heckerelation}
T(m)T(n)=\sum_{d\mid (m,n)}T(mn/d^2).
\end{equation}
Again if $n$ is coprime with non-split primes from \eqref{Heckerelation} one immediately concludes that $T(n)$ can be written as a polynomial of $\{T(p)\}_{p\mid n}$.

\subsection{Construction of the Amplifier}
The construction of the amplifier is inspired by Lemma $22$ in \cite{N}. Let $\lambda_\phi(m)$ be the eigenvalue of a Hecke eigenform $\phi$ for the operator $T(m)$. We will often drop the subscript $\phi$ when it is clear which eigenform is being considered.

We will now choose our amplifier. We define
\begin{equation}\label{amplifier}
K_L=\prod_{p\in \P}\left(\sum_{n=1}^L\lambda_\phi(p^n)T(p^n)\right)^2.
\end{equation}
\begin{lemma}\label{amplification}
\begin{enumerate}
\item
Let $M_{\P,L}=\{M=\prod_{p\in\P}p^{k_p}\mid 1\le k_p\le L\}$. For any integer $L\ge 1$
$$K_L=\sum_{m,n\in\M(\P,L)}\lambda_\phi(m)\lambda_\phi(n)\sum_{d\mid(m,n)}T(mn/d^2).$$
\item
Any $\P$-Hecke eigenform is also a $K_L$ eigenform with non-negative eigenvalue. The $K_L$ eigenvalue of $\phi$ is $\gg L^{2|\P|}$.
\end{enumerate}
\end{lemma}

To prove Lemma \ref{amplification} we first prove the main technical part.

\begin{lemma}\label{naiveamplification}
Let $p$ be an unramified prime. For any Hecke eigenform $\phi$ with $T(m)$-eigenvalue $\lambda(m)$,
$$\sum_{n=1}^L|\lambda(p^n)|^2\gg L,$$
where the implied constant may depend on $p$.
\end{lemma}

The implied constant can be proved to be absolute but that is unnecessary for our purpose as our set of primes is, a priori, fixed. This lemma is standard in the subject, for instance see Lemma $8.3$ of \cite{L}. We record a proof here for the sake of completeness.
\begin{proof}
The proof of this lemma is similar to that of Lemma $22$ in \cite{N} where the author has proved almost the same bound for the eigenvalues of the sphere operators. We will first assume that $\phi$ is tempered at $p$. In this case it is convenient  to parametrize $\lambda(p)=\alpha+\alpha^{-1}$ for some $\alpha\in \C^{(1)}$. Then from the recurrence, 
$$\lambda(p^{n+1})=\lambda(p^n)\lambda(p)-\lambda(p^{n-1}),\quad n\ge 1$$
which follows from \eqref{Heckerelation}, we can conclude by induction that
$$\lambda(p^n)=\frac{\alpha^{n+1}-\alpha^{-n-1}}{\alpha-\alpha^{-1}}.$$ 

We fix $p$. We want to prove that $\sum_{n\le L} |\lambda_{\phi}(p^n)|^2 \gg L$ uniformly in $\phi$. On the contrary, let us assume that there exists a sequence of $(\phi,L)=(\phi_j,L_j)$, with $L\to \infty$ as $j\to \infty$, such that $\sum_{n\le L} |\lambda_{\phi}(p^n)|^2= o(L)$ as $j\to \infty$.
We use the notation $\lll$ (and $\ggg$) as in \cite{N}. For quantities $X$ and $Y$ which depend on $j$ implicitly, by $X\lll Y$ (or $X=o(Y)$) we mean $\overline{\lim}_{j\to\infty}X_j/Y_j=0$. $X\ggg Y$ is defined symmetrically. From the assumption above we will deduce a contradiction. By passing to subsequences there will be two cases to consider, as follows.

If $|\alpha^2-1|\ggg 1/L$ then,
\begin{align*}
\sum _{n\le L}|\lambda(p^n)|^2
&=\frac{2}{|\alpha^2-1|^2}\sum_{n\le L}(1-\Re(\alpha^{2n+2}))\\
&=\frac{2}{|\alpha^2-1|^2}\left(L-\Re\left(\sum_{n\le L}\alpha^{2n+2}\right)\right)\\
&\ge\frac{2}{|\alpha^2-1|^2}\left(L-\frac{2}{|\alpha^2-1|}\right)\\
&=\frac{2}{|\alpha^2-1|^2}(L+o(L))\\
&\gg L+o(L)\gg L.
\end{align*}

On the other hand, if $|\alpha^2-1|\ll 1/L$ then letting $m$ be the largest positive integer such that $m\le L$ and $m|\alpha^2-1|\le 1/10$ we see that $m\gg L$. Thus
$$\sum_{n=1}^L|\lambda(p^n)|^2\ge |\lambda(p^m)|^2=\frac{2(1-\Re(\alpha^{2m+2}))}{|\alpha^2-1|^2}.$$
Letting $\alpha=e^{i\theta}$ the above quantity equals to
$$\left(\frac{\sin((m+1)\theta)}{\sin\theta}\right)^2\gg m^2,$$
which concludes the proof in the tempered case.

Proving the non-tempered case is easy and we proceed as in Lemma 8.3 of \cite{L}. We write $\sigma(p^m)$ for the eigenvalue of $p^{m/2}T(p^m)-p^{m/2-1}T(p^{m-2})$ with assumptions that $T(p^{-1})=0$ and $T(1)=1$.
Thus $$p^{m/2}\lambda(p^{m})=\sum_{k\le m;k\equiv m(2)}\sigma(p^k).$$
Letting $\cosh\theta=\left|\frac{\sigma(p)}{2p^{1/2}}\right|=|\lambda(p)|$ and using the Hecke relation that
$$T(p^n)T(p)=T(p^{n+1})+T(p^{n-1})$$
by induction on $n$ one can easily verify that 
$$\sum_{k=0}^n\sigma(p^{2k})=p^n\frac{\sinh((2n+1)\theta)}{\sinh(\theta)}\ge (2n+1)p^n.$$
Therefore, $\lambda(p^{2n})\gg n$. Similarly one can prove, for any $n$, that $\lambda(p^n)\gg n.$
Thus obviously,
$$\sum_{n=1}^L|\lambda(p^n)|^2\gg L^3\gg L,$$
concluding the proof.
\end{proof}

\begin{proof}[Proof of Lemma \ref{amplification}]
Expanding \eqref{amplifier} and using the \eqref{Heckerelation}
we immediately see $(1)$:
\begin{align*}
K_L
&=\left(\prod_{p\in \P}\left(\sum_{n=1}^L\lambda(p^n)T(p^n))\right)\right)^2=\left(\sum_{m\in\M(\P, L)}\lambda(m)T(m)\right)^2\\
&=\sum_{m,n\in\M(\P,L)}\lambda(m)\lambda(n)\sum_{d\mid(m,n)}T(mn/d^2).
\end{align*}

From the definition of $K_L$ non-negativity of eigenvalues is clear as all Hecke operators, being self-adjoint, have real eigenvalues. The eigenvalue of $\phi$ is
$$\prod_{p\in \P}\left(\sum_{n=1}^L|\lambda(p^n)|^2\right)^2\gg L^{2|\P|},$$
where in the last inequality we use Lemma \ref{naiveamplification}.
\end{proof}
Here we prove a technical lemma which we will be using in the final phase of proving the main theorem.
\begin{lemma}\label{technical}
For $L\ge 1$ and $x\in \R$,
$$\sum_{m,n\in \M(\P,L)}|\lambda(m)\lambda(n)|\sum_{d\mid (m,n)}\left(\frac{\sqrt{mn}}{d}\right)^{x}\ll_{x,\P}
\begin{cases}
\prod_{p\in\P}\sum_{r=1}^L|\lambda(p^r)|^2 &\text{ if } x<0,\\
\prod_{p\in \P}p^{xL}L\sum_{r=1}^L|\lambda(p^r)|^2&\text{ if }x\ge 0.
\end{cases}
$$
\end{lemma}

\begin{proof}
First we note that the LHS is multiplicative, that is, it factors over $p\in \P$. So it is enough to prove the lemma assuming that $\P=\{p\}$. Hence For $x<0$ we get that,
\begin{align*}
& \sum_{1\le r,s\le L}\lambda(p^r)\lambda(p^s)\sum_{i=0}^{\min(r,s)}p^{\left(\frac{r+s}{2}-i\right)x}\\
&\ll \sum_{r=1}^L\sum_{s=1}^L |\lambda(p^r)\lambda(p^s)|p^{x|r-s|/2}\\
&\ll \sum_{l=0}^{L-1}p^{xl/2}\sum_{r=1}^{L-l}|\lambda(p^r)\lambda(p^{r+l})|)\\
&\ll \sum_{l=0}^{L-1}p^{xl/2}\left(\sum_{r=1}^{L-l}|\lambda(p^r)|^2\right)^{1/2}\left(\sum_{r=1}^{L-l}|\lambda(p^{r+l})|^2\right)^{1/2}\\
&\ll\sum_{r=1}^L|\lambda(p^r)|^2.
\end{align*}
If $x\ge 0$ following a similar calculation we see that the sum
\begin{align*}
&\ll\sum_{r=1}^L\sum_{s=1}^L|\lambda(p^r)\lambda(p^s)|p^{(r+s)x/2}\\
&\ll p^{xL}\left(\sum_{r=1}^L|\lambda(p^r)|\right)^2\ll p^{xL}L\left(\sum_{r=1}^L|\lambda(p^r)|^2\right),
\end{align*}•
completing the proof.
\end{proof}

\section{Revisiting B\'erard's Theorem on a Compact Quotient}
In this section we will revisit B\'erard's theorem where a logarithm improvement of the error term in Weyl's eigenvalue counting formula is established on a negatively curved, connected and compact Riemannian $n$-dimensional manifold with boundary. Most of the section follows \S$3$ of \cite{So}, and \S$3$ of \cite{BS}. Recall that $\Delta$ is the Laplace-Beltrami operator of $M$. We label the spectrum of $-\Delta$, counted with multiplicity, as $$0=\lambda_0<\lambda^2_1\leq\lambda^2_2\leq\dots.$$Suppose that $\{\phi_j\}_{j=0}^\infty$, containing $\phi$, is an orthonormal basis of $L^2(M)$, such that, $-\Delta\phi_j=\lambda_j^2\phi_j$. H\"ormander \cite{Ho} proved an asymptotic for the eigenvalue counting function: for a constant $C_M$ depending only on $M$,
\begin{equation}\label{weyl}
N(\lambda):=\#\{\lambda_j\leq\lambda\}=C_M\lambda^n+O(\lambda^{\frac{n-1}{2}}).
\end{equation}
B\'erard \cite{Be} proved that if $M$ has non-positive sectional curvature then one can improve \eqref{weyl} to 
\begin{equation} \label{berard}
N(\lambda):=\#\{\lambda_j\leq\lambda\}=C_M\lambda^n+O\left(\frac{\lambda^{\frac{n-1}{2}}}{\log\lambda}\right).
\end{equation}
We will employ the same argument as in \cite{Be} to improve the bound for the archimedean kernel. In our context, we fix $M=\Gamma\backslash\hh^2$ and $n=2$. We define the operator $E_j$ on $C^\infty(M)$ as
$$E_jf(x)=\phi_j(x)\langle f,\phi_j\rangle=\phi_j(x)\int_M f(y)\overline{\phi_j(y)}d\mu.$$We recall the standard spectral projector with $E_j$'s, which is
\begin{equation}\label{specproj}
h_\lambda f=\sum_{\lambda_j\in[\lambda,\lambda+1]}E_jf.
\end{equation}
For the purpose of the proof of \eqref{berard} we need to change the spectral window in \eqref{specproj}, which is done as follows. By $\hat{h}$ we will denote the Fourier transform of a function $h$.
Let $h\in\mathcal{S}(\R)$ be a Schwartz function such that
\begin{equation}\label{propertyh}
h(0)=1, h\ge 0, \text{ and } \hat{h}(t)=0 \text{ for } |t| \ge 1/2.
\end{equation}
Such a function exists, as one can cook up $h$ by taking the Fourier transform of the convolution of an even non-negative function $0\neq\chi\in C_c^\infty((-1/4,1/4))$ with itself. Scaling the spectral window by a factor $\epsilon >0$, which will be fixed later, we define
\begin{equation}\label{finalspecproj}
\tilde{h}^\epsilon_\lambda f=\sum_{j=0}^\infty h(\epsilon^{-1}(\lambda-\lambda_j))E_jf.
\end{equation}
We also define a remainder operator by $$r_\lambda f=\sum_{j=0}^\infty h(\epsilon^{-1}(\lambda+\lambda_j))E_jf.$$ Then by the Fourier transform we get that,
\begin{align*}
\tilde{h}^\epsilon_\lambda f+r_\lambda f
&=\frac{1}{2\pi}\sum_{j=0}^\infty\epsilon\left(\int_\R\hat{h}(\epsilon t)e^{-it\lambda_j}e^{it\lambda}dt+\int_\R\hat{h}(\epsilon t)e^{it\lambda_j}e^{it\lambda}dt\right)E_jf\\
&=\frac{1}{\pi}\int_\R\epsilon\hat{h}(\epsilon t)e^{it\lambda}\left(\sum_{j=0}^\infty\cos(t\lambda_j)E_jf\right)dt.
\end{align*}
Hence equating the kernels of these operators we rewrite \eqref{finalspecproj} as
\begin{equation}\label{ptf}
\sum_{j=0}^\infty \left(h(\epsilon^{-1}(\lambda-\lambda_j))+ h(\epsilon^{-1}(\lambda+\lambda_j))\right)\phi_j(x)\overline{\phi_j(y)} = \frac{1}{\pi}\int_\R\epsilon\hat{h}(\epsilon t)e^{it\lambda}\cos t\sqrt{-\Delta_M} (x,y) dt.
\end{equation}
Here as a kernel function of two variables one can write
$$\cos t\sqrt{-\Delta_M}(x,y)=\sum_{j=0}^\infty\cos(t\lambda_j)\phi_j(x)\overline{\phi_j(y)}.$$
Recall the covering map $p$ at $m\in M$; $p=\exp_m\colon T_mM\to M$. We pull back $f\in C^\infty(M)$ by $p$ to $\tilde{f}\in C^\infty(\hh^2)$ which is periodic with respect to $\Gamma$. Pulling the metric on $M$ back to $\hh^2$ and folding the right hand side we may rewrite the kernel as
\begin{equation}\label{folding}
(\cos t\sqrt{-\Delta_M})(\tilde{z},\tilde{w})=\sum_{\gamma\in\Gamma}(\cos t\sqrt{-\Delta_{\hh^2}})(\gamma z,w),
\end{equation}
where $p(z)=\tilde{z}$, $p(w)=\tilde{w}$ and $\Delta_{\hh^2}$ is the Laplacian on $\hh^2$. Combining with \eqref{ptf} we get that
\begin{equation}\label{fptf comp}
\sum_{j=0}^\infty h(\epsilon^{-1}(\lambda-\lambda_j))\phi_j(z)\overline{\phi_j(w)}+\sum_{j=0}^\infty h(\epsilon^{-1}(\lambda+\lambda_j))\phi_j(z)\overline{\phi_j(w)} = \sum_{\gamma\in\Gamma}K_\epsilon(\gamma z,w) dt,
\end{equation}
where
$$K_\epsilon(z,w):=\frac{1}{\pi}\int_\R\epsilon\hat{h}(\epsilon t)e^{it\lambda}\cos t\sqrt{-\Delta_{\hh^2}} (z, w).$$
This equation is the so-called pre-trace formula. By $d_g$ we will denote the usual Riemannian distance on $\hh^2$ with respect to its hyperbolic metric. We also note that if $d_g(z,w)$ is sufficiently small ($\le 3$ should be enough, see \S$3$ of \cite{BS}) then
\begin{equation}\label{huygens}
\cos(t\sqrt{-\Delta_{\hh^2}})(z,w)=\cos(t\sqrt{-\Delta_M})(\tilde{z},\tilde{w}),
\end{equation} due to Huygens principle.

\begin{lemma}\label{trivbndB}
There exists an absolute constant $C>0$ such that for $\lambda^{-1} \le \epsilon\le 1$,
$$K_\epsilon(z,w)\ll
\begin{cases}
&\lambda\epsilon+e^{C/\epsilon},\text{ if }d_g(z,w)\le 1\\
&\left(\frac{\lambda}{d_g(z,w)}\right)^{1/2}e^{C/\epsilon},\text{ if $d_g(z,w)\le\epsilon^{-1}$}; \\
&0, \text{ otherwise}.
\end{cases}$$
\end{lemma}
Most of the proof of this Lemma already exists in the proof of Proposition $3.6.2$ of \cite{So}, however not as explicit as in the statement of this Lemma. We provide a complete proof using the ingredients which have been developed in the proof of Proposition $3.6.2$ of \cite{So} (we also refer $\S3$ of \cite{BS}) for the convenience of the reader.
\begin{proof}
We claim the following:
there exists an absolute constant $C>0$ such that for $\lambda^{-1} \le \epsilon\le 1$,
\begin{equation}\label{mainclaim}
K_\epsilon(z,w)\ll
\begin{cases}
&\lambda\epsilon+e^{C/\epsilon}\text{ if $z=w$};\\
&\left(\frac{\lambda}{d_g(z,w)}\right)^{1/2}e^{C/\epsilon},\text{ if $0<d_g(z,w)\le\epsilon^{-1}$}\\
&0,\text{ otherwise}.
\end{cases}
\end{equation}
We will first see that the claim \eqref{mainclaim} is sufficient for the proof.
Note that $h\ge 0$. We define
$$X_j(z):=\sum_{j=0}^\infty \sqrt{(h(\epsilon^{-1}(\lambda-\lambda_j))+h(\epsilon^{-1}(\lambda+\lambda_j)))}\phi_j(z).$$
Let us fix $z,w\in \hh^2$ with $d_g(z,w)\le 1$. Then using \eqref{huygens} and Cauchy-Schwarz we get that
\begin{align*}
&K_\epsilon(z,w)=\int_\R\epsilon\hat{h}(\epsilon t)e^{it\lambda}\cos t\sqrt{-\Delta_{M}}(\tilde{z},\tilde{w})dt\\
&=\left|\sum_{j=0}^\infty X_j(\tilde{z})\overline{X_j(\tilde{w})}\right| 
\le \left(\sum_{j=0}^\infty |X_j(\tilde{z})|^2\right)^{1/2}\left(\sum_{j=0}^\infty |X_j(\tilde{w})|^2\right)^{1/2}
\le\sup_z K_\epsilon(z,z).
\end{align*}
Thus from the first estimate of \eqref{mainclaim} we conclude the first estimate of Lemma \ref{trivbndB}. Rest of the estimates of Lemma \ref{trivbndB} are same as in the claim \eqref{mainclaim}.

Now to prove the claim \eqref{mainclaim} we follow as in the proof of Proposition $3.6.2$ of \cite{So}. Among the three estimates of \eqref{mainclaim} proofs of first two estimates are similar. The first estimate, however, has been proved exactly as stated in \eqref{mainclaim}, in the proof of $(3.6.8)$ of \cite{So}. So we focus on showing the second estimate of \eqref{mainclaim}. 

We may prove the lemma using theory of Hadamard's parametrix for the solution of the Cauchy problem. For details about parametrix we refer to $\S2.4$ of \cite{So}. We consider the kernel function on $(\R^2,g)$ instead of the upper half plane, where $g$ is the metric that makes $(\R^2,g)$ isometric to $\hh^2$. From $(3.6.10)$ and $(1.2.30)$ of \cite{So} we write 
$$\cos(t\sqrt{-\Delta_{\hh^2}})(x,y)=\sum_{m=0}^N\alpha_m(x,y)\partial_t(E_m(t,d_g(x,y))-E_m(-t,d_g(x,y)))+R_N(t,x,y),$$
where $E_m$ are radial tempered distributions supported on $\{(t,x)\in[0,\infty)\times\R^2\mid d_g(x,0)\le t\}$, defined by (as $E_m$'s are radial in second variable, by abuse of notation, we are identifying the second entry of $E_m$ as the distance from origin)
\begin{align*}
E_m(t,x)&\equiv E_m(t,d_g(x,0))\\
&=\lim_{\epsilon\to 0_+}m!(2\pi)^{-3}\int_\R\int_{\R^{2}}e^{i\langle x,\xi\rangle+it\tau}(d_g(\xi,0)^2-(\tau-i\epsilon)^2)^{-m-1}d\xi d\tau.
\end{align*}
From $(3.6.10)-(3.6.13)$ and Theorem $3.1.5$ of \cite{So} we obtain following properties of the coefficients $\alpha_m$ and the remainder term $R_N$:
$\alpha_m$ are smooth functions on $M\times M$ with 
\begin{equation}\label{estimatealpha}
0<\alpha_0(x,y)\le 1, \alpha_0(x,x)=1\text{ and }|\partial_{x,y}^l\alpha_m|\ll e^{O(d_g(x,y))},\quad |l|, m = O(1),
\end{equation}
and there is a $\delta>0$ such that 
$$R_N\in \mathcal{C}^{N-5}([-\delta,\delta]\times M\times M)\text{ with }|\partial^\alpha_{t,x,y}R_N|\ll |t|^{2N-|\alpha|}e^{O(t)},\quad N\ge 5.$$
From the estimate of $R_N$ we conclude that
\begin{equation}\label{estimateR}
\int_\R\epsilon\hat{h}(\epsilon t)e^{it\lambda}R_N(t, x,y)dt \ll e^{O(\epsilon^{-1})}.
\end{equation}
Now from $(3.6.15)$ (and its proof which also works for $r>0$) we get that
\begin{equation}\label{estimatehint}
\int_\R\hat{h}(\epsilon t)e^{it\lambda}(\partial_tE_m(t,r)-\partial_tE_m(-t,r))dt\ll (\lambda/r)^{1/2-m}.
\end{equation}
Note that \eqref{estimatehint}, along with \eqref{estimatealpha} for $d_g(x,y)\le \epsilon^{-1}$ and \eqref{estimateR}, proves the second estimate of \eqref{mainclaim}. To see that we write
\begin{align*}
K_\epsilon(z,w)
&\ll \sum_{m=0}^N|\alpha_m(x,y)|\left|\int_\R\hat{h}(\epsilon t)e^{it\lambda}\partial_t(E_m(t,d_g(x,y))-E_m(-t,d_g(x,y)))\right|dt\\
&\hphantom{estimateestimate}+\left|\int_\R\epsilon\hat{h}(\epsilon t)e^{it\lambda}R_N(t,x,y)\right|dt\\
&\ll e^{C\epsilon^{-1}}\sum_{m=0}^N (\lambda/d_g(x,y))^{1/2-m}+e^{C\epsilon^{-1}},
\end{align*}
for some absolute constant $C$ and $\epsilon\le\lambda^{-1}$. This completes proof of the second estimate.

The third inequality of \eqref{mainclaim} is follows from the support of $\hat{h}$ as described in \eqref{propertyh} and the fact that the Cauchy kernel $\cos(t\sqrt{-\Delta_{\hh^2}})(x,y)$ vanishes when $d_g(x,y)> t$ (for details see discussion in p. $46$ of \cite{So}).
\end{proof}

\section{Counting and Bounds}
\subsection{Estimating Hecke Returns}
We recall the definition of standard point-pair invariant. Thinking of the upper half plane as $\SL_2(\R)/\mathrm{SO}(2)$, we recognize a point $x+iy=z\in \hh^2$ as $gK$, where $K=\mathrm{SO}(2)$. From the Iwasawa decomposition we write $g=\begin{pmatrix}
\sqrt{y}&x\sqrt{y^{-1}}\\&\sqrt{y^{-1}}
\end{pmatrix}K$. We define a point-pair invariant by
$$u(z,i)=\frac{x^2+(y-1)^2}{4y}.$$
One may also note that $u(z,w)=\sinh^2(d_g(z,w)/2)$.
Fix $z\in M$. For $t>0$ an $N\in\mathbb{N}$ we define a counting function $$M(N,t;z)=\#\{\gamma\in R(N)\mid u(\gamma z,z)<t\}.$$

\begin{lemma}\label{heart}
For any $t>1$ $$M(N,t;z)\ll t N^{2}.$$
Also there exists a $\delta_N>0$ such that $M(N,\delta_N;z)\ll_\nu N^\nu$ for all $\nu>0$ small enough.
\end{lemma}
\begin{proof}
This is direct from Lemma $1.3$ of \cite{IS} which says that for any $t> 0$
$$M(N,t;z)\ll ((t+t^{1/4})N+1)N^\nu.$$
which immediately implies the first claim. For
\begin{equation}\label{choosedelta}
t=\delta_N=N^{-4},
\end{equation}
the second claim is also true.
\end{proof}

\subsection{Final Bound}
In this section we will fix an $h$ as in  \eqref{propertyh} and $$\epsilon=c(\log\lambda)^{-1}$$ in Lemma \ref{trivbndB} where $c>C$ is a constant to be fixed later. We will also drop the subscript $\epsilon$ from $K_\epsilon$.
 
\begin{proof}[Proof of Theorem \ref{theorem1}] We choose an orthonormal Hecke-eigenbasis containing $\phi$ in the pre-trace formula in \eqref{fptf comp}. We apply our amplifier $K_L$ in \eqref{amplifier} to \eqref{fptf comp} in the $z$ variable and then plug-in $z=w$. Using \eqref{propertyh} and Lemma \ref{amplification} and thus using non-negativity of each term of the left hand side of amplified pre-trace formula we deduce that
\begin{equation}\label{amplified ptf}
A_L^2|\phi(z)|^2\ll \sum_{m,n\in \M(\P,L)}\sum_{d\mid (m,n)}\frac{d}{\sqrt{mn}}\left|\lambda_\phi(m)\lambda_\phi(n)\sum_{\gamma\in R(mn/d^2)} K(\gamma z,z)\right|,
\end{equation}
where $$A_L:=\prod_{p\in\P}\sum_{n=1}^L |\lambda_\phi(p^n)|^2$$ and thus from Lemma \ref{naiveamplification} $A_L\gg L^{|\P|}$. 

Recall that $u=\sinh^2(d_g/2)$. We note that 
$$u >\lambda^{1/c}\implies d_g>\frac{\log \lambda}{c}=\epsilon^{-1}.$$
Thus from Lemma \ref{trivbndB} we get that 
$$K(\gamma z, z)=0 \text{ if } u(\gamma z ,z)>\lambda^{1/c},$$
and so for any natural number $N$ and small enough $\delta>0$,
\begin{align}\label{splitting}
\begin{split}
&\sum_{\gamma\in R(N)}K(\gamma z,z)= \sum_{\substack{{\gamma\in R(N)}\\ {u(\gamma z,z)<\delta}}}K(\gamma z, z)+\sum_{\substack{{\gamma\in R(N)}\\ {\delta\le u(\gamma z,z)\le\lambda^{1/c}}}}K(\gamma z, z)\\
& \ll \frac{\lambda}{\log \lambda}M(N,\delta;z)+\left(\frac{\lambda}{\log(1+2\delta)}\right)^{1/2}\lambda^{C/c}M(N,\lambda^{1/c};z).\\
\end{split}
\end{align}
We have used Lemma \ref{trivbndB} in the last line.
Now we apply Lemma \ref{heart} with $\nu=1/10$ and $\delta=(\sqrt{mn}/d)^{-8}$, and $N=mn/d^2$. From the chosen values of the parameters we obtain estimates that
$$\frac{\lambda^{1/2}}{\log(1+2\delta)^{1/2}}\ll \lambda^{1/2}\delta^{-1/2}=\lambda^{1/2}\left(\frac{\sqrt{mn}}{d}\right)^4;$$
and from \eqref{choosedelta}
$$M(mn/d^2,\delta;z)\ll \left(\frac{\sqrt{mn}}{d}\right)^{1/5};\quad M(mn/d^2,\lambda^{1/c};z)\ll\lambda^{1/c}\left(\frac{\sqrt{mn}}{d}\right)^4.$$
Now we combine \eqref{amplified ptf} and \eqref{splitting}, and use $\delta=(\sqrt{mn}/d)^{-8}$ in the $d$th summand of \eqref{amplified ptf}, along with above estimates. Thus we obtain that
\begin{align*}
&A_L^2|\phi(z)|^2\\
&\ll \sum_{m,n\in\M(\P,L)}\sum_{d\mid (m,n)}|\lambda_\phi(m)\lambda_\phi(n)\frac{d}{\sqrt{mn}}|\left(\frac{\lambda}{\log\lambda}\left(\frac{\sqrt{mn}}{d}\right)^{1/5}+\lambda^{1/2+C/c}\left(\frac{\sqrt{mn}}{d}\right)^{4}\lambda^{1/c}\left(\frac{\sqrt{mn}}{d}\right)^4\right)\\
&= \frac{\lambda}{\log\lambda}\sum_{m,n\in\M(\P,L)}|\lambda_\phi(m)\lambda_\phi(n)|\sum_{d\mid (m,n)}\left(\frac{\sqrt{mn}}{d}\right)^{-4/5}\\
&\hphantom{supnormsupnormsupnorm}+\lambda^{1/2+(1+C)/c}\sum_{m,n\in \M(\P,L)}|\lambda_\phi(m)\lambda_\phi(n)|\sum_{d\mid (m,n)}\left(\frac{\sqrt{mn}}{d}\right)^{7}\\
&\ll \frac{\lambda}{\log\lambda}A_L +\lambda^{1/2+(C+1)/c}L^{|\P|}A_L\prod_{p\in \P}p^{7L}.
\end{align*}
We have used Lemma \ref{technical} in the last estimate. Choosing $L=\frac{\log\lambda}{100\log \prod_{p\in\P}p}$ and $c=4C+4$ we conclude that
$$|\phi(z)|^2\ll_\epsilon \frac{\lambda}{(\log\lambda)^{|\P|+1}}+\lambda^{3/4+7/100+\epsilon} \ll \frac{\lambda}{(\log\lambda)^{|\P|+1}},$$which gives the result.
\end{proof}

\subsection{Remark}
One can ask about the optimality of the method; for e.g. whether one would be able to achieve a power saving using linear amplifier with finitely many primes. We give an argument why power of logarithm saving should be optimal.
One can understand from the proof that if we use an amplifier of the form $$(\sum_m \alpha_mT(p^m))^2$$ (the definition of $T(n)$ is in \S $2$) then the geometric side of the amplified pre-trace formula is \textit{of the size} of
$$\sqrt{\frac{\lambda}{\log\lambda}}\times\sum|\alpha_m|^2.$$
Hence to get a saving from \eqref{startB} we need $\frac{(\sum_m\alpha_m\lambda_\phi(p^m))^2}{\sum_m|\alpha_m|^2}$ to be as big as possible, where $\lambda_\phi(p^m)$ is the eigenvalue of $\phi$ under $T(p^m)$. By Cauchy-Schwarz this quotient is the largest when $\alpha_m=\lambda_\phi(p^m)$. On the other hand, the size of $\sum_{m\le L} |\lambda_\phi(p^m)|^2$ is comparable to $L$, at least when $\phi$ is tempered at $p$ such that the saving is restricted to at most a power of logarithm of $\lambda$. 

\begin{ackno} We want to thank Lior Silberman whose supervision during the author's masters thesis helped to start this project and also for encouragement. We  want to thank Paul Nelson for teaching us the amplification method he used in \cite{N} and numerous helpful suggestions on an earlier draft of this paper. We also want to thank Djordje Mili\'cevi\'c, Gergely Harcos, Kevin Nowland, Felix Dr\"axler, and the anonymous referee for several helpful feedbacks, comments, and suggestions on various aspects of the work.
\end{ackno}


\begin{thebibliography}{99}
\bibitem{Be}
B\'erard, P.: \textit{On the wave equation on a compact manifold without conjugate points}, Math. Z. 155 (1977), 249-276.

\bibitem{BHM}
Blomer, V.;  Harcos, G.; Mili\'cevi\'c, D.: \textit{Bounds for eigenforms on arithmetic hyperbolic $3$-manifolds}, Duke Math. J. 165 (2016), 625-659. 

\bibitem{BHMM}
Blomer, V.;  Harcos, G.; Maga, P.; Mili\'cevi\'c, D.: \textit{The sup-norm problem for $\GL(2)$ over number fields}, arXiv:1605.09360.

\bibitem{BM1}
Blomer, V.; Maga, P.: \textit{The sup-norm problem for $\PGL(4)$}, IMRN 2015 (vol. 14), 5311-5332.

\bibitem{BM2}
Blomer, V.; Maga, P.: \textit{Subconvexity for sup-norms of automorphic forms on $\PGL(n)$}, Selecta Math. 22 (2016), 1269-1287.

\bibitem{BP}
Blomer, V.; Pohl, A.: \textit{The sup-norm problem on the Siegel modular space of rank two}, Amer. J. Math. 136 (2016), 999-1027.

\bibitem{BL}
Brooks, S.; Lindenstrauss, E.: \textit{Non-localization of eigenfunctions on large regular graphs}, Israel Journal of Mathematics
January 2013, Volume 193, Issue 1, pp 1-14.

\bibitem{BS}
Blair, M.; Sogge, C.: \textit{Concerning Toponogov's Theorem and logarithmic improvement of estimates of eigenfunctions}, to appear, Journal of Differential Geometry, arXiv:1510.07726.

\bibitem{E} Eichler, M.: \textit{Lectures on Modular Correspondences},  Tata Institute.  (9)1955.

\bibitem{Ho}
Hormander, L.: \textit{The spectral function of an elliptic operator}, Acta Math., 121 (1968), pp.
193–218.

\bibitem{HRR}
Holowinsky, R.; Ricotta, G.; Royer, E.: \textit{On the sup norm of an $\SL(3)$ Hecke-Maass cusp form},	arXiv:1404.3622.

\bibitem{IS}
Iwaniec, H; Sarnak, P.: \textit{$L^\infty$-norms of eigenfunctions of arithmetic surfaces}, Annals of Mathematics Second Series, Vol. 141, No. 2 (Mar., 1995), pp. 301-320.

\bibitem{L}
Lindenstrauss, E. :\textit{Invariant measures and arithmetic quantum unique ergodicity},
Ann. of Math. (2) 163 (2006), no. 1, 165–219.

\bibitem{Ma2}
Marshall, S.: \textit{Sup norms of Maass forms on semisimple groups},	arXiv:1405.7033.

\bibitem{N}
Nelson, P.: \textit{Microlocal lifts and quantum unique ergodicity on $\GL(2,\Q_p)$}, arXiv:1601.02528.

\bibitem{Sa1}
Sarnak, P.: \textit{Letter to Morawetz}, available at http://www.math.princeton.edu/sarnak/.

\bibitem{So}
Sogge, C.: \textit{Hangzhou lectures on eigenfunctions of the laplacian}, Annals of Mathematics Studies, 188. Princeton University Press, Princeton, NJ, 2014. xii+193 pp. ISBN: 978-0-691-16078-8 

\end{thebibliography}
\end{document}